\documentclass[12pt,reqno]{amsart}
\usepackage{amssymb}
\usepackage{amsmath}

\usepackage{amscd}

\newcommand{\RNum}[1]{\uppercase\expandafter{\romannumeral #1\relax}}

\usepackage[T2A]{fontenc}
\usepackage[utf8]{inputenc}
\usepackage[russian,english]{babel}
\input{int.def}

\usepackage{tikz-cd}
\usetikzlibrary{cd}
\usepackage{dirtytalk}
\usepackage{hyperref}
\usepackage{xcolor}
\usepackage{centernot}

\usepackage{amsmath}

\usepackage{enumitem}

\numberwithin{equation}{section}

\usepackage[mathcal]{euscript}

\usepackage{titlesec}
\titleformat{\section}[runin]{\bfseries}{\thesection.}{3pt}{}[.]

\myitemmargin 
\usepackage{geometry}
\newgeometry{vmargin={25mm}, hmargin={22mm,22mm}, footskip=10mm}   

\begin{document}

\title[Anti-self-dual blowups]%
{Anti-self-dual blowups}

\author{Vsevolod Shevchishin}
\address{Faculty of Mathematics and Computer Science\\
University of Warmia and Mazury\\
ul.~Słoneczna 54, 10-710 Olsztyn, Poland
}
\email{vsevolod@matman.uwm.edu.pl, shevchishin@gmail.com}

\author{Gleb Smirnov}
\address{School of Mathematics, 
University of Geneva, Rue du Conseil-G{\'e}n{\'e}ral, 7, 1205, Geneva, Switzerland}

\email{gleb.smirnov@unige.ch}




\begin{abstract}
Let $X$ be a closed, oriented four-manifold containing an embedded sphere with self-intersection number $(-1)$. Suppose that $b_2^+(X) \leq 3$. We show that there exists a Riemannian metric on $X$ such that the cohomology class dual to this sphere is represented by an anti-self-dual harmonic form. Furthermore, such a metric can be constructed even when there are multiple disjoint embedded $(-1)$-spheres.
\end{abstract}

\maketitle

\setcounter{section}{0}
\section{Main result}\label{main}
Let $(X, g)$ be a closed, oriented Riemannian $4$-manifold. The bundle of 2-forms over $X$ decomposes as $\Lambda^+ \oplus \Lambda^-$, where $\Lambda^+$ and $\Lambda^-$ are the eigenspaces of the Hodge star operator $\star: \Lambda^2 \to \Lambda^2$. A section $\varphi$ of $\Lambda^2$ is called self-dual (SD) if $\star \varphi = \varphi$ and anti-self-dual (ASD) if $\star \varphi = -\varphi$. By the Hodge theorem, each cohomology class in $H^2(X; \mathbb{R})$ has a unique harmonic representative, giving a canonical identification:
\[
H^2(X; \rr) = \left\{ \varphi \in \Gamma(\Lambda^2)\ |\ \Delta\, \varphi = 0 \right\}.
\]
The Hodge star commutes with $\Delta$, leading to a direct-sum decomposition: 
\[
H^2(X;\rr) = \calh^{+}_g \oplus \calh^{-}_g,
\]
where $\calh^{\pm}_g$ are defined as follows:
\[
\calh^{\pm}_g = \left\{ \varphi \in \Gamma(\Lambda^{\pm})\ |\ \Delta\, \varphi = 0 \right\}.
\]
The cup product pairing is positive-definite on $\mathcal{H}^+_g$ and negative-definite on $\mathcal{H}^-_g$, with $\mathcal{H}^+_g$ and $\mathcal{H}^-_g$ orthogonal in $H^2(X; \mathbb{R})$. 
The dimensions 
$b^\pm_2 = \operatorname{dim} \mathcal{H}^\pm_g$ are topological invariants: $b^+_2$ is the dimension of the maximal positive-definite subspace, and $b^-_2$ is the dimension of the maximal negative-definite subspace. 
\smallskip%

Assume that \( X \) contains \( n \) disjoint embedded \((-1)\)-spheres \( E_1, \ldots, E_n \). Then \( X \) is smoothly a connected sum:
\[
X = M \# n\,{\overline{\bold{CP}}^{\,2}},
\]
where each \((-1)\)-sphere \( E_i \) is contained in a different copy of \({\overline{\bold{CP}}^{\,2}}\), representing a line \({\overline{\bold{CP}}^{\,1}} \subset {\overline{\bold{CP}}^{\,2}}\). This note proves:
\begin{theorem}\label{t:main}
If $b^+_2 \leq 3$, there exists a Riemannian metric $g$ on $X$ such that for every self-dual harmonic form $u$ on $X$ and each $i$,
\[
\int_{E_i} u = 0.
\]
In other words, if $e_i \in H^2(X; \mathbb{Z})$ is the class Poincar\'e dual to $E_i$, then each $e_i$ is represented by an anti-self-dual form with respect to $g$.
\end{theorem}
\begin{remark}
A Riemannian metric $g$ is called \emph{almost-K{\"a}hler} if there is a symplectic form $\omega$ and an $\omega$-compatible almost-complex structure $J$ such that $g(\cdot, \cdot) = \omega(\cdot, J\cdot)$. In this case, $\omega$ is a $g$-self-dual harmonic form. If $\omega$ is a symplectic form on $X$, and $e \in H^2(X; \mathbb{Z})$ is the class dual to an embedded $(-1)$-sphere, Taubes shows in \cite{Taub-2} that $[\omega] \cup e \neq 0$. Thus, the metric provided by Theorem \ref{t:main} cannot be almost-K{\"a}hler.
\end{remark}
Let $\operatorname{Met}(X)$ be the space of all smooth Riemannian metrics on $X$, equipped with the $C^{\infty}$-topology. Define $\Omega$ as the set of all $b^+_2$-dimensional subspaces of $H^2(X; \mathbb{R})$ where the cup product is positive-definite; this is an open subset of the Grassmannian of $b^+_2$-dimensional subspaces of $H^2(X; \mathbb{R})$. There is a canonical map:
\[
P \colon \operatorname{Met}(X) \to \Omega,
\]
assigning to each $g \in \operatorname{Met}(X)$ the subspace $\mathcal{H}^+_g \subset H^2(X; \mathbb{R})$. 
This map is called the \emph{period map} of $X$, as introduced in \cite{LeBr-1}. The concept of a period map arises in K3 surface theory: if $X$ is the underlying smooth manifold of a complex K3 surface and we restrict to the space of Kähler metrics, then $P$ corresponds to the (Burns-Rapoport) period map for K3 surfaces (see \cite{Kob-Tod}, \cite{B-R}, \cite{Siu-2}, and references therein).  
\smallskip%

The period map $P$ is invariant under the action of the Torelli group of $X$, which consists of all diffeomorphisms of $X$ acting trivially on its homology. It can also be shown that $P(g)$ depends only on the conformal class of $g$. Currently, there is no comprehensive theory of Riemannian period maps, and fundamental questions, such as whether $P$ is surjective or whether its fibers are connected, remain open. However, it is known that $P$ is an open map (in the $C^{\infty}$-topology). 
\smallskip%

To rephrase Theorem \ref{t:main} in terms of the period map: let $e_1, \ldots, e_n \in H^2(X; \mathbb{Z})$ be the cohomology classes Poincar\'e dual to the spheres $E_1, \ldots, E_n$, respectively. If $b^{+}(X) \leq 3$, there exists a Riemannian metric $g$ on $X$ such that $P(g) \perp e_i$ for each $i$.
\smallskip%

One motivation for studying the map $P$ is that the set of cohomology classes of negative square represented by ASD forms is, in principle, a smooth invariant of the manifold. In some cases, this invariant could be stronger than Seiberg-Witten invariants.
\smallskip%

Let $L \to X$ be a line bundle over $X$ with a Hermitian inner product, and set $c_1(L) = e$. Consider a Riemannian metric on $X$. A connection $A$ on $L$ is called anti-self-dual if its curvature satisfies $F^{+}_A = 0$, meaning the cohomology class $e$ is represented by an anti-self-dual form. Anti-self-dual connections on 
$\UU(1)$-bundles (i.e., line bundles) are the simplest examples of instantons; they correspond to reducible solutions to the Seiberg-Witten equations and play a fundamental role in the theory. In his foundational works \cite{Taub-3, Taub-4}, Taubes studied the existence of instantons on $\SU(2)$-bundles, and this note initiates the study of the $\UU(1)$ case. 
\smallskip%

A significant difference between our work and \cite{Taub-3, Taub-4} is that Taubes's existence results are stable under deformations of the Riemannian metric, while $\UU(1)$-instantons generally disappear under arbitrarily small perturbations of the metric. Indeed, if $g$ is a metric on $X$ such that $L$ admits an anti-self-dual connection, then necessarily $P(g) \perp e$. Since $P$ is open, any small perturbation of $P(g)$ in $\Omega$ still lies in the range of $P$. Thus, if $b^{+}(X) > 0$, there exists a small perturbation $g'$ of $g$ such that $P(g')$ is not orthogonal to $e$. Consequently, $L$ does not admit an anti-self-dual connection with respect to $g'$. 

\section{Sketch proof of Theorem \ref{t:main}}\label{sketch}
We provide a brief sketch of the proof of Theorem \ref{t:main}. The formal proof is given in \S\,\ref{main_proof}. Let $X$ be given as:
\[
X = M \# {\overline{\bold{CP}}^{\,2}},
\]
where $b^{+}_2(M) \leq 3$. In \S\,\ref{main_proof}, we explain how to adapt the proof for the case of multiple $(-1)$-spheres. We assume that $b^{+}_2(M) = 3$, as the case of lower $b^{+}_2$ is simpler.  
\smallskip%

To begin, we equip \( M \) with a special Riemannian metric, the existence of which is proven in \S\,\ref{main_proof}. A crucial step in this proof relies on the condition that \( b_2^+ \) does not exceed 3 — the dimension of \( \Lambda^+ \). Let $p \in M$ be a point on $M$, and let $g$ be a metric on $M$ such that a neighborhood $U$ of $p$ in $M$ is isometric to the unit ball in the standard flat $\mathbb{R}^4$. Let $g_0$ denote that flat metric on $\mathbb{R}^4$. It will be convenient to regard $\mathbb{R}^4$ as $\mathbb{C}^2$ and consider $g_0$ as a K{\"a}hler flat metric. We identify $(U,g)$ with the unit ball in $(\mathbb{C}^2, g_0)$, and when referring to the restriction of $g$ to $U$, we use the notation $g_0$. Let $r \colon M \to \mathbb{R}^2$ be the distance function from $p \in M$. We write $\omega = -d\,d^{\mathbb{C}} r^2$ for the K{\"a}hler form on $U$ associated with $g_0$. Since $b^{+}_2 = 3$, $M$ admits three independent SD harmonic forms $\psi_1, \psi_2, \psi_3$. Let $\langle \psi_i, \omega \rangle_p$ denote the inner product of $\psi_i$ and $\omega$ evaluated at $p$. We assume that $g$ is chosen so that for each $i$, 
\[
\langle \omega, \psi_i \rangle_p = 0.
\]
For $0 < a_1 < a_2 < 1$, consider the spherical shell
\[
N_{a_1, a_2} = \left\{ x \in U \mid a_1 \leq r(x) \leq a_2 \right\}.
\]
Let $(S,g_{S})$ be isometric to the unit sphere in $(\mathbb{C}^2, g_0)$. Using the coordinate $r$, we can write the metric $g_0$ on $N_{a_1, a_2}$ as:
\[
dr \otimes dr + r^2\, g_{S}.
\]
Setting $r = e^{-t}$, we can also express $g_0$ as $e^{-2t}\,dt \otimes dt + e^{-2t}\, g_{S}$. Thus, $N_{a_1, a_2}$ is conformally isometric to $[0,T] \times S^3$, with $T = \ln{a_2} - \ln{a_1}$, endowed with the product metric 
$dt \otimes dt + g_{S}$.
\smallskip%

Choose $a > 0$ and let $N \subset U$ be defined as:
\[
N = \left\{ x \in U \mid a \leq r(x) \leq a \cdot e \right\},\quad \ln{e} = 1.
\]
Apply a conformal transformation to $g$ so that $g$ remains unchanged outside a small neighborhood of $N$ and such that $(N,g)$ becomes isometric to $[0,1] \times S$ with the product metric $dt \otimes dt + g_{S}$. We will still denote the transformed metric by $g$. Note that the forms $\psi_i$ are still harmonic SD forms with respect to the transformed metric.
\smallskip%

Let $X$ be the blowup of $M$ at $p$, and let $C \subset X$ be the corresponding exceptional $(-1)$-sphere. A standard technique from \cite{G-H} is used to build a metric $g'$ on $X$ that agrees with $g$ outside an arbitrarily small neighborhood of $p$. More details on $g'$ are provided in \S\,\ref{blowup}.
\smallskip%

Let $X_{T}$ be the Riemannian manifold obtained from $(X, g')$ by replacing $N$ with $[0,T] \times S^3$. A gluing procedure in \S\,\ref{neck} associates to every harmonic SD form $\psi$ on $M$ a harmonic SD form $u_{T}$ on $X_{T}$ that is $C^{k}$-close to $\psi$ on $M - U$ for sufficiently large $T$. To simplify notation, we drop the subscript $T$ from $u_{T}$ and use $u$ for all forms $u_{T}$ defined on 
different Riemannian manifolds $X_{T}$. In \S\,\ref{blowup}, we prove the following formula:
\[
\int_{C} u = A \cdot \langle \omega, \psi \rangle_{p}\, e^{-2T} + O(e^{-cT}),
\]
where $A > 0$ and $c > 2$ are constants independent of $T$.
\smallskip%

To proceed, we endow $M$ with a special 3-dimensional family of metrics $g_s$, $s \in D^3$, the existence of which is proven in \S\,\ref{main_proof}. For $s = 0$, $g_s$ agrees with $g$ on the entirety of $M$. All $g_s$ agree with $g$ on $U$. $(M, g_{s})$ carries SD harmonic forms $\psi_{1s}$, $\psi_{2s}$, $\psi_{3s}$ that depend smoothly on $s$. Associated with the family $g_s$, there is an evaluation map $\pi \colon D^3 \to \mathbb{R}^3$ defined as:
\begin{equation}\label{eval_1}
\pi(s) = \left( 
\langle \omega, \psi_{1s} \rangle_p, 
\langle \omega, \psi_{2s} \rangle_p,
\langle \omega, \psi_{3s} \rangle_p
\right).
\end{equation}
We choose $g_s$ such that $\pi$ maps the point $s = 0$ in $D^3$ to the origin of $\mathbb{R}^3$ and such that $\pi$ is a local diffeomorphism from a neighborhood of $0 \in D^3$ onto a neighborhood of the origin in $\mathbb{R}^3$.
\smallskip%

Extend $g_{s}$ to a family of metrics $g'_{s}$ on the blowup of $M$ at $p$, that is, to $X$. Since each $g_s$ agrees with $g$ on $U$, we perform this extension in the same way we extended $g$ to $g'$. Next, replace $N$ with $[0,T] \times S$ to obtain a family of metrics, parameterized by $s \in D^3$, on $X_{T}$. Applying the gluing procedure in \S\,\ref{neck} to $\psi_{is}$, we obtain SD harmonic forms $u_{is}$ on $X_{T}$; then we have:
\[
\left( \int_C u_{1s}, 
\int_C u_{2s}, 
\int_C u_{3s} \right) = 
A \cdot \pi(s) e^{-2T} + O(e^{-cT}).
\]
It follows that for large enough $T$, there exists an $s \in D^3$ such that the left-hand side of the above equality vanishes. This completes the proof.

\section{A neck-stretching argument}\label{neck}
This section presents material known to experts (see, e.g., \cite{Atiyah-Pat-Sing-1, Nic-2, Taub-6}) in a form suited to our application. Let $X_1$ and $X_2$ be oriented compact 4-manifolds with boundary $Y$. Choose a Riemannian metric $g_1$ on $X_1$ such that a small closed neighborhood $P_1$ of $\partial X_1$ in $X_1$ is isometric to $[0,1] \times Y$ with the product metric:
\begin{equation}\label{product-metric}
dt \otimes dt + g_Y,
\end{equation}
where $t \in [0,1]$ and $g_Y$ is a metric on $Y$. Further, let $g_1$ be such that for some $\delta > 0$, an open neighborhood of $P_1$ in $X_1$ is isometric to $(-\delta, 1] \times Y$ with the same metric \eqref{product-metric}, where $P_1 = [0,1] \times Y \subset (-\delta, 1] \times Y$. Similarly, choose $g_2$ and $P_2$ on $X_2$.
\smallskip%

For convenience, we assume $Y$ is the 3-sphere with the round metric of radius 1, but the results hold if $H^1(Y; \mathbb{R}) = 0$.
\smallskip%

Define the non-compact elongation $X_i^{\infty}$ of $X_i$ by attaching the semi-infinite cylinder $[0, \infty) \times Y$ with the metric \eqref{product-metric}, identifying (isometrically) the cuff $P_i$ with $[0, 1] \times Y$. Let $g_i^{\infty}$ denote the metric on $X_i^{\infty}$. 
\smallskip%

Next, glue $X_1$ and $X_2$ together: choose $T > 0$ and define the \textit{neck} $N_T = [0, T] \times Y$ with the product metric. Consider $N_T$ as part of a larger cylinder $N'_T = [-1, T + 1]$. Identify (isometrically) $P_1$ with $[-1, 0] \times Y$ and $P_2$ with $[T, T+1] \times Y$. The resulting manifold, denoted $X_T$, consists of three pieces: $X_1, N_T, X_2$, with metrics $g_T$ such that $(X_1, g_T) = (X_1, g_1)$, $(X_2, g_T) = (X_2, g_2)$, and $(N_T, g_T)$ is isometric to $[0,T] \times Y$. For different values of $T$, we still refer to $N_T$ simply as $N$.
\smallskip%

Define two subsets of $N$:
\[
Q_1 = [0,1] \times Y, \quad Q_2 = [T-1, T] \times Y.
\]
Let $\mathcal{D} = d + d^{*}$ be the Hodge operator and $\Delta = \mathcal{D} \circ \mathcal{D}$ the Hodge Laplacian. We denote by $\mathcal{H}^+(X_i^{\infty}) = \operatorname{ker} \Delta$ the space of $L^2$-class SD harmonic 2-forms on $X_i^{\infty}$. For a domain $A$, $\|\cdot\|_{k, A}$ denotes the $L^2_k$-Sobolev norm on $A$.
\smallskip%

Let $\rho_2$ be a smooth non-negative function on $X_T$ that is zero on $X_1 - P_1$, transitions smoothly from 0 to 1 in $P_1$, and equals 1 away from $X_1$. Similarly, define $\rho_1$ on $X_T$ to be zero on $X_2 - P_2$, transitions smoothly in $P_2$, and equals 1 away from $X_2$.
\smallskip%

For $\psi \in \mathcal{H}^+(X_1^{\infty})$ with $\mathcal{D}\psi = 0$, we introduce two sequences of SD 2-forms, $v^{(i)}$ and $u^{(i)}$:
\[
\text{
$v^{(i)} \in L^2(X_1^{\infty})$ if $i$ is odd}, \quad
\text{
$v^{(i)} \in L^2(X_2^{\infty})$ if $i$ is even}, \quad
u^{(i)} \in L^2(X_{T}), \quad 
i = 1,2,\ldots.
\]
These forms depend on $T$, though this dependence is not explicit in our notation. Define $v^{(i)}$ and $u^{(i)}$ inductively as follows:

\begin{enumerate}[label=(\alph*)]
\item $v^{(1)} = \psi$.
\smallskip%

\item If $i$ is odd, restrict $v^{(i)}$ to $X_1 \cup N \cup P_2$ and set $u^{(i)} = v^{(i)} \rho_1$. If $i$ is even, restrict $v^{(i)}$ to $X_2 \cup N \cup P_1$ and set $u^{(i)} = v^{(i)} \rho_2$.
\smallskip%

\item Assume $i$ is odd; the even case is handled similarly. We construct a canonical solution $v^{(i)}$ to the equation $\mathcal{D}\, v^{(i)} = -\mathcal{D}\, u^{(i - 1)}$ on $X_1^{\infty}$. The 2-form $v^{(i)}$ will belong to $L^2(X_1^{\infty})$ and satisfy the estimate:
\begin{equation}\label{vestim}
\| v^{(i)} \|_{k, X_1 \cup Q_1} \leq A_k \| \Delta\, u^{(i - 1)} \|_{k - 2, P_1}
\end{equation}
for some constants $A_k$ independent of $T$ and $i$. Assuming $u^{(i-1)}$, $v^{(i-1)}$ have been constructed, construct $v^{(i)}$ as follows:
\end{enumerate}
Fix two numbers $T_1$ and $T_2$ such that $T_2 > T_1 > 0$, and define a strictly positive function $\chi \colon [0, +\infty) \times Y \to \mathbb{R}$ with $\chi = 1$ on $[0, T_1] \times Y$ and $\chi = e^{-2t}$ on $[T_2, +\infty) \times Y$. Extend $\chi$ to all of $X_1^{\infty}$ by setting $\chi = 1$ on $X_1$. Endow $X_1^{\infty}$ with the metric $\chi\, g_1^{\infty}$, giving it an asymptotically Euclidean structure: the semi-infinite cylinder $[T_2, +\infty) \times Y$ in $(X_1^{\infty}, \chi\, g_1^{\infty})$ is isometric to the closed ball of radius $e^{-T_2}$ in flat $\mathbb{R}^4$ minus the origin. Next, compactify $(X_1^{\infty}, \chi\, g_1^{\infty})$ by adding a point at infinity, denoting the resulting compactification by $(\hat{X}_1, \hat{g}_1)$.
\smallskip%

$\mathcal{D}\, u^{(i - 1)}$ is supported within the union of $P_1$ and $P_2$. Restrict $\mathcal{D}\, u^{(i - 1)}$ to $X_1 \cup N$. Since $\mathcal{D}\, u^{i}$ vanishes on $N$, we extend it to the entirety of $\hat{X}_1$ by setting it to zero outside $X_1$.
\smallskip%

Let $\mathcal{H}(\hat{X}_1)$ be the space of harmonic 2-forms on $(\hat{X}_1, \hat{g}_1)$. We have: 
\[
\int_{\hat{X}_1} \langle \varphi, \Delta\, u^{(i-1)}\rangle = 
\int_{\hat{X}_1} \langle \mathcal{D}\,\varphi, \mathcal{D}\, u^{(i-1)}\rangle = 0, \quad
\text{for each $\varphi \in \mathcal{H}(\hat{X}_1)$,}
\]
where $\langle \cdot, \cdot \rangle$ denotes the pointwise inner product. This holds because $\Delta\,\varphi = 0$ implies $\mathcal{D}\,\varphi = 0$. 
\smallskip%

Since $\Delta\, u^{(i-1)} \perp \mathcal{H}(\hat{X}_1)$, there exists a unique SD 2-form $\hat{v}^{(i)}$ on $(\hat{X}_1, \hat{g}_1)$ such that $\hat{v}^{(i)} \perp \mathcal{H}(\hat{X}_1)$ and $\mathcal{D}\,\hat{v}^{(i)} = - \mathcal{D}\,u^{(i-1)}$. Using standard elliptic estimates, we obtain: 
\[
\| \hat{v}^{(i)} \|_{k, \hat{X}_1} \leq 
A_k \| \Delta\, u^{(i - 1)} \|_{k - 2, P_1}
\]
for some constants $A_k$ depending only on $\hat{g}_1$. Let $v^{(i)}$ be the restriction of $\hat{v}^{(i)}$ to $X_1^{\infty}$. Choosing $T_2 > 0$ larger if needed, we may arrange that $\chi = 1$ on $Q_1$; then $g_1^{\infty}$ and $\hat{g}_1$ agree on $X_1 \cup Q_1$. Thus:
\[
\| v^{(i)} \|_{k, X_1 \cup Q_1} \leq \| \hat{v}^{(i)} \|_{k, \hat{X}_1}.
\]
Here, the norm on the left is computed using $g_1^{\infty}$, which implies \eqref{vestim}.
\smallskip%

It is clear that $v^{(i)}$ is $L^2$ on $(X_1^{\infty}, \chi\, g_1^{\infty})$. Since the $L^2$-norm of a 2-form is conformally invariant, it follows that $v^{(i)}$ is also $L^2$ on $(X_1^{\infty}, g_1^{\infty})$. Moreover, $v^{(i)}$ satisfies the equation $\mathcal{D}\,v^{(i)} = -\mathcal{D}\,u^{(i-1)}$ on $(X_1^{\infty}, \chi\, g_1^{\infty})$. We claim this equation remains valid when $\chi\, g_1^{\infty}$ is replaced with $g_1^{\infty}$. To see this, note that $\mathcal{D}\,u^{(i-1)}$ is supported in $P_1$, and the kernel of $\mathcal{D}$ is \emph{locally} conformally invariant. Thus, we only need to verify the equation at points in $P_1$. Since $\chi$ is chosen such that $g_1^{\infty}$ and $\chi\, g_1^{\infty}$ agree on $P_1$, the claim follows.
\smallskip%

With $u^{(i)}$ and $v^{(i)}$ constructed, we formally obtain:
\[
\mathcal{D}\, \left( \sum_{i = 1}^{\infty} u^{(i)} \right) = 0.
\]
More precisely, for each $i$, the support of $\Delta\,u^{(i)}$ is confined to $P_1 \cup P_2$. Additionally, $\mathcal{D}\,u^{(i)} = -\mathcal{D}\,u^{(i-1)}$ on $P_2$ for even $i$, and $\mathcal{D}\,u^{(i)} = -\mathcal{D}\,u^{(i-1)}$ on $P_1$ for odd $i$.

\begin{lemma}\label{iteration}
Given $k > 0$, the series
\[
u = \sum_{i = 1}^{\infty} u^{(i)}
\]
converges on $(X_{T}, g_{T})$ in the $L^2_{k}$-norm for all sufficiently large $T$. Moreover, 
\[
\| u - u^{(1)} \|_{k, X_{T}} = O(e^{-2T}),\quad 
\| u - u^{(1)} - u^{(2)} \|_{k, X_{T}} = O(e^{-4T}).
\]
\end{lemma}
\begin{proof}
It suffices to find constants $C_k$ such that 
\[
\| u^{(i)} \|_{k, X_{T}} \leq C_k\,e^{-2T} \| u^{(i-1)} \|_{k, X_{T}}.
\]
To this end, we derive several auxiliary estimates. In the neck $N$, $\Delta$ takes the form (see, e.g., \cite{Atiyah-Pat-Sing-1}):
\[
\Delta = -\partial^2_{t} + \Delta_{3},
\]
where $\Delta_{3}$ is the Hodge Laplacian on $Y$, independent of $t$. Since $\Delta\, \psi = 0$, $\psi$ admits a Fourier expansion on $N$:
\begin{equation}\label{psi-exp}
\psi = \sum_{\lambda} 
(\alpha_{\lambda} + \star \alpha_{\lambda}) e^{-\lambda\,t},
\end{equation}
where $\lambda$ runs over the positive and negative square roots of the eigenvalues of $\Delta_{3}$. Each coefficient $\alpha_{\lambda}$ is an eigenform of $\Delta_{3}$ associated with $\lambda^2$, and $\star \alpha_{\lambda}$ is obtained by applying the Hodge star operator to $\alpha_{\lambda}$.
\smallskip%

The eigenforms and spectrum of the Laplacian $\Delta_3$ on $p$-forms on the round sphere $S^n$ are known (see \cite{Foll}). For our case, $n = 3$ and $p = 2$, the minimal eigenvalue is 3, followed by 4. However, the condition $\cald\,\psi = 0$ implies that 
$d \alpha_{\lambda} = 0$ and that 3 will never occur; therefore, the minimal eigenvalue is 4. More generally, the eigenvalues of $\Delta_3$ for \emph{closed} eigenforms are $(2+k)^2$, where $k = 0, 1, 2, \ldots$.  
\smallskip%

Since $\psi$ is $L^2$ on $X_1^{\infty}$, the summation in \eqref{psi-exp} includes only positive $\lambda$'s:
\[
\psi = \sum_{\lambda \geq 2} 
(\alpha_{\lambda} + \star \alpha_{\lambda}) e^{-\lambda t}.
\]
Let $Q_1 + s \subset N$ be the subset defined by $Q_1 + s = [s, s + 1] \times Y$. Consider the $L^2$-norm of $\psi$ on $Q_1 + s$:
\[
\| \psi \|_{0, Q_1 + s}^{2} = \sum_{\lambda \geq 2} \int_{s}^{s + 1} 2\, e^{-2 \lambda t} \int_{Y} 
\langle \alpha_{\lambda}, \alpha_{\lambda} \rangle.
\]
Since $\alpha_{\lambda}$ do not depend on $t$, we get:
\begin{equation}\label{psi-l2}
\| \psi \|_{0, Q_1 + s}^{2} \leq e^{-4\,s} \| \psi \|_{0, Q_1}^{2}.
\end{equation}
Fix a small $\delta > 0$. The region 
$Q_1 + s$ lies in the union of $Q_1 + s + \delta$ and $Q_1 + s - \delta$. 
Using the standard elliptic estimates, we get: 
\begin{equation}\label{psi-sobolev}
\| \psi \|_{k, Q_1 + s}^2 \leq B'_k \left( 
\| \psi \|_{0, Q_1 + s + \delta}^2 + \| \psi \|_{0, Q_1 + s - \delta}^2 
\right).
\end{equation}
where $B'_k$ is independent of $s$ since $\Delta$ is a translation-invariant operator. Note that \eqref{psi-sobolev} is valid for each $s \in [0, T]$, as we have arranged for $g_{T}$ to remain the product metric in a slightly larger $\delta$-neighborhood of $P_2 = Q_1 + T$.
\smallskip%

Combining \eqref{psi-l2} with \eqref{psi-sobolev}, we get:
\[
\| \psi \|_{k, Q_1 + s}^2 \leq e^{-4\,s} B_{k} \| \psi \|_{0, Q_1}^2 \leq e^{-4\,s} B_{k} \| \psi \|_{k, Q_1}^2
\]
We can now apply this estimate twice:
\[
\| \psi \|_{k, N}^2 = \| \psi \|_{k, Q_1}^{2} + 
\| \psi \|_{k, Q_1 + 1}^{2} + \ldots + 
\| \psi \|_{k, Q_1 + T - 1}^2 \leq D_k \| \psi \|_{k, Q_1}^2,
\]
and
\[
\| \psi \|_{k, P_2}^2 = \| \psi \|_{k, Q_1 + T}^2 \leq e^{-4\,T} B_{k} \| \psi \|_{k, Q_1}^2.
\]
The constants $B_k$ and $D_k$ do not depend on $\psi$ and are the same for 
all $v^{(i)}$'s. Thus, we get:
\begin{equation}\label{v-N}
\| v^{(i)} \|_{k, N}^2 \leq 
D_{k} \| v^{(i)} \|_{k, Q_1}^2\ \text{for $i$ odd,}\quad 
\| v^{(i)} \|_{k, N}^2 \leq 
D_{k} \|v^{(i)} \|_{k, Q_2}^2\ \text{for $i$ even,}
\end{equation}
and
\begin{equation}\label{u-P12}
\| v^{(i)} \|_{k, P_2}^2 \leq 
e^{-4\,T} B_{k} \| v^{(i)} \|_{k, Q_1}^2\ \text{for $i$ odd,}\quad 
\| v^{(i)} \|_{k, P_1}^2 \leq 
e^{-4\,T} B_{k} \| v^{(i)} \|_{k, Q_2}^2\ \text{for $i$ even.}\quad
\end{equation}
Since $u^{(i)}$ and $v^{(i)}$ agree on $N$, we get from \eqref{v-N} that:
\begin{equation}\tag{A}\label{A}
\| u^{(i)} \|_{k, N}^2 \leq 
D_{k} \| u^{(i)} \|_{k, Q_1}^2\ \text{for $i$ odd,}\quad 
\| u^{(i)} \|_{k, N}^2 \leq 
D_{k} \|u^{(i)} \|_{k, Q_2}^2\ \text{for $i$ even.}
\end{equation}
Since 
$\rho_1$ and $\rho_2$ are smooth, we have: 
\[
\| u^{(i)} \|_{k, P_2}^2 \leq E_k \| v^{(i)} \|_{k, P_2}^2\ 
\text{for $i$ odd,}\quad 
\| u^{(i)} \|_{k, P_1}^2 \leq E_k \| v^{(i)} \|_{k, P_1}^2\ 
\text{for $i$ even.}\quad
\]
Here, the constants $E_k$ do not depend on $i$ and are determined by $\rho_1$ and $\rho_2$. 
Combining this with \eqref{u-P12} we find:
\[
\| u^{(i)} \|_{k, P_2}^2 \leq e^{-4\,T} E_{k} B_{k} \| v^{(i)} \|_{k, Q_1}^2\  
\text{for $i$ odd,}\quad 
\| u^{(i)} \|_{k, P_1}^2 \leq e^{-4\,T} E_{k} B_{k} \| v^{(i)} \|_{k, Q_2}^2\ 
\text{for $i$ even.}\quad
\]
$u^{(i)}$ and $v^{(i)}$ agree on $Q_1$ for each odd $i$, and agree 
on $Q_2$ for each even $i$. Setting $F_k = E_{k} B_{k}$, we get:
\begin{equation}\tag{B}\label{B}
\| u^{(i)} \|_{k, P_2}^2 \leq e^{-4\,T} F_{k} \| u^{(i)} \|_{k, Q_1}^2\  
\text{for $i$ odd,}\quad 
\| u^{(i)} \|_{k, P_1}^2 \leq e^{-4\,T} F_{k} \| u^{(i)} \|_{k, Q_2}^2\ 
\text{for $i$ even.}\quad
\end{equation}
Using continuity of $\Delta$, we write \eqref{vestim} as:
\begin{equation}\tag{C}\label{C}
\| u^{(i)} \|_{k, X_1 \cup Q_1} \leq 
A_k \| \Delta\, u^{(i - 1)} \|_{k - 2, P_1} 
\leq G_{k} \| u^{(i - 1)} \|_{k, P_1}.
\end{equation}
Let $i$ be odd. We calculate:
\begin{multline*}
\|u^{(i)}\|^2_{k, X_{T}} = 
\|u^{(i)}\|^2_{k, X_1} + 
\|u^{(i)}\|^2_{k, N} + 
\|u^{(i)}\|^2_{k, P_2} 
\stackrel{\eqref{A}}{\leq} 
\|u^{(i)}\|^2_{k, X_1} + 
D_k \|u^{(i)}\|^2_{k, Q_1} + 
\|u^{(i)}\|^2_{k, P_2} 
\stackrel{\eqref{B}}{\leq}\\
\leq 
\|u^{(i)}\|^2_{k, X_1} + 
D_k \|u^{(i)}\|^2_{k, Q_1} +
e^{-4\,T} F_k \|u^{(i)}\|^2_{k, Q_1} \leq 
H_{k} \|u^{(i)}\|^2_{k, X_1 \cup Q_1} 
\stackrel{\eqref{C}}{\leq}\\ 
\leq H_{k} G_k^2 \|u^{(i-1)}\|^2_{k, P_1} 
\stackrel{\eqref{B}}{\leq} 
e^{-4\,T} F_k H_{k} G_k^2 \|u^{(i-1)}\|^2_{k, Q_2} 
\leq e^{-4\,T} C_k^2 \|u^{(i-1)}\|^2_{k, X_T}.
\end{multline*}
A similar argument applies when $i$ is even. This completes the proof. \qed
\end{proof}
\smallskip%

Consider again the Fourier expansion of $\psi$: 
\begin{equation}\label{psi-four}
\psi = \sum_{\lambda\, \geq\, 2} \psi_{\lambda},\quad 
\psi_{\lambda} = a_{\lambda} e^{-\lambda\,t},
\end{equation}
and write $v^{(2)}$ as follows: 
\[
v^{(2)} = \sum_{\lambda\,\geq\,2} v^{(2)}_{\lambda},
\]
where $v^{(2)}_{\lambda} \in L^2(X_2^{\infty})$ 
are solutions to the following equations:
\[
\cald\, v^{(2)}_{\lambda} = - \cald\, \left( \rho_1 \psi_{\lambda} \right).
\]
The solutions $v^{(2)}_{\lambda}$ are obtained 
in the same way as $v^{(2)}$. 
Put $u^{(2)}_{\lambda} = v^{(2)}_{\lambda}\rho_2$. 
Our goal is to show 
that $u^{(2)}$ is well approximated by $u^{(2)}_{2}$.
\begin{lemma}\label{u2}
$\| \sum_{\lambda\,>\, 2} u^{(2)}_{\lambda} \|_{k, X_2} = O(e^{-c\,T})$ for 
some $c > 2$.
\end{lemma}
\begin{proof}
Choose $c > 0$ such that $c^2$ is the next largest eigenvalue 
of $\Delta_3$ after 4 (exact value of $c$ is 3, but it's not needed for our purpose). Since 
\[
\cald\, \left( \sum_{\lambda\,\geq\, c} u^{(2)}_{\lambda} \right) = 
- \cald\, \left( \rho_1 \sum_{\lambda\,\geq\, c} \psi_{\lambda} \right),
\]
we get:
\[
\| \sum_{\lambda\,\geq\, c} u^{(2)}_{\lambda} \|_{k, X_2} 
\stackrel{\eqref{C}}{\leq}
G_k \| \rho_1 \sum_{\lambda\,\geq\, c} \psi_{\lambda} \|_{k, P_2}
\leq 
G_k \sqrt{E_k} \| \sum_{\lambda\,\geq\, c} \psi_{\lambda} \|_{k, P_2}.
\]
Using an estimate similar to \eqref{psi-l2}, we continue:
\[
\| \sum_{\lambda\,\geq\, c} \psi_{\lambda} \|_{k, P_2} 
\leq 
e^{-c\,T} B_k \| \sum_{\lambda\,\geq\, c} \psi_{\lambda} \|_{0, Q_1} = O(e^{-c\,T}),
\]
and the lemma follows. \qed
\end{proof}

\section{Blowup formula}\label{blowup}
Let $(M, g)$ be a Riemannian 4-manifold, and let \(p \in M\) be a point. Choose the metric \(g\) so that a neighborhood \(U\) of \(p\) in \(M\) is isometric to the unit ball in \(\mathbb{R}^4\) with the standard flat metric \(g_0\). Choose a linear complex structure on \(\mathbb{R}^4\) such that \(g_0\) is K{\"a}hler. Thus, \(U\) can be considered a complex domain, where the metric \(g\) is K{\"a}hler. When referring to the restriction of \(g\) to \(U\), we often use \(g_0\). Denote by \(\omega\) the K{\"a}hler 2-form on \(U\). Let \(r \colon M \to \mathbb{R}\) be the distance function from \(p\). Then on \(U\), we have \(\omega = -d\,d^{\mathbb{C}} r^2\). In terms of \(r\), \(U = \{ x \in X \mid r(x) < 1 \}\). The group $\UU(2)$, which preserves the Hermitian metric \(g_0\) on \(\mathbb{C}^2\), acts by isometries on \(U\). Both the distance function \(r\) and the 2-form \(\omega\) are invariant under this action.
\smallskip%

Let \(g_{S}\) be the metric on the unit sphere \(S \subset (\mathbb{C}^2, g_0)\). The flat metric \(g_0\) can then be expressed as:
\[
g_0 = dr \otimes dr + r^2\, g_{S}.
\]
Setting \(r = e^{-t}\), this becomes:
\[
g_0 = e^{-2t}\,dt \otimes dt + e^{-2t}\, g_{S}.
\]
Thus, the punctured ball \((U, g)\) is conformally isometric to the semi-infinite cylinder \((0, +\infty) \times S\) with the product metric \(dt \otimes dt + g_{S}\).
\smallskip%

Choose \(a\) such that \(0 < a < a \cdot e < 1\), and define \(N \subset U\) as:
\[
N = \{ x \in U \mid a \leq r(x) \leq a \cdot e \}.
\]
In terms of the semi-infinite cylinder, \(U\) corresponds to the region \([- \ln{a} - 1, -\ln{a}] \times S\).
\smallskip%

To apply Lemma \ref{iteration}, we introduce the following notations:
\[
P_1 = \{ x \in U \mid a \cdot e \leq r(x) \leq a \cdot e^2 \},\quad 
P_2 = \{ x \in U \mid a \cdot e^{-1} \leq r(x) \leq a \}.
\]
These notations are consistent with those used in \S\,\ref{neck}.
\smallskip%

Let \(\psi\) be a harmonic SD form on \((M, g)\). In the neighborhood \(U \subset M\), the 2-form \(\psi\) admits a Fourier expansion:
\[
\psi = \sum_{\lambda \geq 2} \psi_{\lambda},\quad 
\psi_{\lambda} = a_{\lambda} r^{\lambda}.
\]
Since \(\psi\) is a harmonic SD form with respect to \(g_0\), it is also harmonic SD with respect to the product metric \(e^{-2t} g_0\). Substituting \(r = e^{-t}\), we obtain the series:
\[
\psi = \sum_{\lambda \geq 2} 
a_{\lambda} e^{-\lambda t},
\]
which is of the form \eqref{psi-four}. The summation includes only positive \(\lambda\) since \(\psi\) is defined over all of \(M\). Each \(\psi_{\lambda} = a_{\lambda} e^{-\lambda t}\) is a harmonic SD 2-form on \(U\).
\smallskip%

If \(\lambda > 2\), the \(g_0\)-length of \(\psi_{\lambda}\) tends to 0 as \(r \to 0\), so:
\[
\langle \psi, \omega \rangle_{p} = \langle \psi_2, \omega \rangle_{p},
\]
where the inner product \(\langle \psi, \omega \rangle_{p}\) is taken with respect to \(g_0\).
\smallskip%

Let \(\mathbb{C}^2\) (and \(U \subset \mathbb{C}^2\)) have complex coordinates \((z, w)\). Set \(\Omega = dz \wedge dw\). The eigenspace of \(\Delta_3\) associated with the eigenvalue 4 is 3-dimensional, and we can write \(\psi_2\) as:
\[
\psi_2 = a_1\,\omega_1 + a_2\,\omega_2 + a_3\,\omega_3,\quad a_i \in \mathbb{R},\qquad 
\text{where}\quad 
\omega_1 = \frac{\omega}{4},\ 
\omega_2 = \operatorname{Re} \Omega,\ 
\omega_3 = \operatorname{Im}\Omega.
\]
\(\langle \omega_i, \omega_j \rangle\) are constant on \(U\) and satisfy \(\langle \omega_i, \omega_j \rangle = 2 \delta_{ij}\). Hence, 
\(\langle \psi_2, \omega \rangle = 8 a_1\). 
\smallskip%

Choose a strictly positive function \(\chi \colon M \to \mathbb{R}\) such that \(\chi = 1\) outside a small neighborhood of \(P_1 \cup N \cup P_2\) and \(\chi = e^{-2t}\) on \(P_1 \cup N \cup P_2\). Then \((N, \chi g)\) is isometric to \([0,1] \times S\) with the product metric \(dt \otimes dt + g_{S}\). Since \(\chi g\) and \(g\) are conformal, the form \(\psi\) remains harmonic SD, as do the forms \(\psi_{\lambda}\) and \(\omega_i\). 
\smallskip%

Let \(X\) be the blowup of \(M\) at \(p\). Denote by \(C \subset X\) the \((-1)\)-sphere arising from the blowup, and let \(\sigma \colon X \to M\) be the blow-down map of \(C\). The function \(r \circ \sigma \colon \sigma^{-1}(M) \to \mathbb{R}\) is smooth and plurisubharmonic on \(\sigma^{-1}(U)\), but not strictly plurisubharmonic. Choose a small \(\delta > 0\) and define \(U_{\delta} \subset M\) as \(U_{\delta} = \{ x \in U \mid r(x) < \delta \}\). A classical result (see, e.g., Ch.\,1 in \cite{G-H}) states that \(X\) admits a function \(h\) that is strictly plurisubharmonic 
on \(\sigma^{-1}(U)\) and satisfies \(h = r^2\) on \(X - \sigma^{-1}(U_{\delta})\). Let \(g'\) be the Riemannian metric on \(\sigma^{-1}(U)\) induced by the K{\"a}hler 2-form \(\omega' = -d\,d^{\mathbb{C}} h\). The metrics \(g'\) and \(g\) agree outside \(\sigma^{-1}(U_{\delta})\).
\smallskip%

Now multiply \(g'\) by the function \(\chi\). Let \(X_{T}\) be the manifold obtained from \((X, \chi g')\) by replacing the region \(N\) with \([0,T] \times S\). Define \(X_1 \subset M\) as:
\[
X_1 = \{ x \in M \mid r(x) \geq a \cdot e \}.
\]
Since \(X_1\) does not contain \(N\), we consider \(X_1\) as a subset of \(X_T\). Similarly, define \(X_2 \subset X_{T}\) as:
\[
X_2 = X_{T} - N - X_1.
\]
\begin{lemma}\label{blowup-lemma}
Let \(M\) and \(X_{T}\) be as above. Suppose there exists a harmonic SD form \(\psi\) on \(M\). Let \(\langle \psi, \omega \rangle\) be the function on \(U\) given by the pointwise inner product with respect to \(g\), and let \(\langle \psi, \omega \rangle_{p}\) denote its value at \(p\). Then, for all sufficiently large \(T\), there exists a harmonic SD form \(u\) on \(X_{T}\) such that \(\| u - \psi \|_{C^{0},\, X_1} \to 0\) as \(T \to \infty\), and
\begin{equation}\label{blowup-area}
\int_{C} u = A \cdot \langle \omega, \psi \rangle_{p}\, e^{-2T} + O(e^{-cT})
\end{equation}
for some constants \(A > 0\) and \(c > 2\), independent of \(T\). 
\end{lemma}
\begin{proof}
Since \(\psi\) is defined on \(X_1 \subset M\), it is also defined on \(P_1\). As \(\psi\) is harmonic with respect to \(\chi g'\), we can analytically continue \(\psi\) through \(N\) to \(P_2\). On \(P_2\), \(\psi\) has the expansion:
\[
\psi = \sum_{\lambda \geq 2} \psi_{\lambda} e^{-\lambda T}.
\]
Let \(u\) be the harmonic SD form on \(X_T\) obtained by applying Lemma \ref{iteration}. Then:
\[
u = \rho_1 \psi + u^{(2)} + O(e^{-4T}).
\]
Restricting \(u\) to \(X_2\), we use Lemma \ref{u2} to get:
\[
u^{(2)} = u^{(2)}_2 + O(e^{-cT}),\quad 
\text{where } \mathcal{D}\, u^{(2)}_2 = -e^{-2T} \mathcal{D}\,(\rho_1 \psi_2).
\]
Here, \(\mathcal{D}\) is with respect to \(\chi g'\). To solve the last equation, we show that \(\psi_2\) can be extended from \(P_2\) to a closed SD form on all of \(X_2\). The pullback of \(\Omega\) by the blow-down map \(\sigma\) is a closed holomorphic form on \(X_2\), also denoted by \(\Omega\). Since \(\Omega\) is of type \((2,0)\), it is self-dual with respect to the metric \(g'\) on \(X_2\). Additionally, the K{\"a}hler form \(\omega'\) agrees with \(\omega\) on \(P_2\). Since both \(\omega'\) and \(\Omega\) are closed SD forms for \(g'\), they are also so for \(\chi g'\). Thus, the form \(\eta = 4^{-1} a_1 \omega' + a_2 \operatorname{Re} \,\Omega + a_3 \operatorname{Im} \,\Omega\) matches \(\psi_2\) on \(P_2\), and we find:
\[
u^{(2)}_2 = e^{-2T} (1 - \rho_1) \eta.
\]
In a small neighborhood of \(C\), \(\rho_1 = 0\), giving:
\[
u = e^{-2T} \eta + O(e^{-cT}),\quad 
\int_{C} u = e^{-2T} \frac{\langle \psi_2, \omega \rangle}{32} \int_{C} \omega' + O(e^{-cT}).
\]
This completes the proof. \qed
\end{proof}

\section{Proof of Theorem \ref{t:main}}\label{main_proof}
To simplify the exposition, we assume \(n = 1\) and \(b^{+}_2 = 3\), and explain how to adapt the proof for \(n > 1\). Let \((M, g)\) be a Riemannian 4-manifold containing an open subset \(V\) that is isometric to an open subset of flat \(\mathbb{R}^4\). We choose \(g\) more carefully: it can be arranged such that \(g\) is still flat in a smaller neighborhood \(V' \subset V\) and that \((M, g)\) has a harmonic SD form \(\psi_1\) with zeros in \(V'\). This result was proved by Taubes in \cite{Taub-5}, where the near-symplectic form found in Proposition 2.1 of \cite{Taub-5} has a \emph{circle} of zeros. Therefore, we assume \(M\) has a flat neighborhood \(V \ni p\) and a harmonic SD form \(\psi_1\) such that \(\psi_1(p) = 0\). For \(n > 1\), we choose points \(p_1, \ldots, p_n\) on the circle of zeros of \(\psi\) and proceed similarly.
\smallskip%

One can choose a complex structure \(J\) on \(V\) such that \(g\) is K{\"a}hler on \(V\). Let \(r\) be the distance function from \(p \in V\), and set \(\omega = -d\,d^{\mathbb{C}} r^2\). We choose \(J\) 
more carefully: let \(\psi_2\) and \(\psi_3\) be SD harmonic forms on \(M\). It can be arranged that \(\psi_2(p)\) and \(\psi_3(p)\) are orthogonal to \(\omega(p)\). This is where the condition $b_{2}^{+} \leq 3$ is used.
\smallskip%

Consider a family of Riemannian metrics \(g_{s}\) on \(M\), with \(s \in D^3\) (the unit 3-disk), such that \(g_0 = g\). Choose triples of SD harmonic 2-forms \(\{\psi_{i s}\}\), for \(i = 1,2,3\), defined on \((M, g_s)\), that depend smoothly on \(s \in D^3\) and satisfy \(\{\psi_{i s}\} = \{\psi_{i}\}\) when \(s = 0\). For each \(i = 1,2,3\), define:
\[
\pi_i(s) = \langle \omega, \psi_{i s} \rangle_{p}.
\]
By construction, \(\pi_i(0) = 0\). 

\begin{lemma}\label{brun}
There exists a family of Riemannian metrics \(g_s\) on \(M\), with \(s \in D^3\), and a small neighborhood \(U\) of \(p\) in \(X\), such that \(g_0 = g\), each \(g_s\) agrees with \(g\) on \(U\), and the mapping \((\pi_1, \pi_2, \pi_3) \colon D^3 \to \mathbb{R}^3\) constructed as above is a local diffeomorphism from a neighborhood of \(g\) in \(D^3\) onto a neighborhood of the origin in \(\mathbb{R}^3\).
\end{lemma}

\begin{proof}
We closely follow the proof of Proposition 1 in \cite{LeBr}. Denote by \(\mathcal{H}^{\pm}_g\) the SD and ASD parts of the space of \(g\)-harmonic 2-forms on \(M\). The conformal class of \(g\) is uniquely determined by the Hodge star involution \(\star \colon \Lambda^2 \to \Lambda^2\). Given a volume form, a family of involutions \(\star_s\) of \(\Lambda^2\) with \(\star_0 = \star\), \(s \in D^3\), uniquely determines a family of metrics \(g_s\) on \(M\).

Let such a family \(\star_s\) be given. Denote by \(a_i \in H^2(M; \mathbb{R})\) the cohomology class of \(\psi_i\). For each \(s \in D^3\), let \(\psi_{i s}\) be the SD part of the harmonic form representing \(a_i\). Clearly, \(\psi_{i s}\) depends smoothly on \(s\). Choose local coordinates \(s_1, s_2, s_3\) on \(D^3\). For each \(j\), let \(h_j = \partial_{s_j} \star_{s} \vert_{s = 0}\). The three fields of endomorphisms \(h_j\) uniquely determine \(\partial_{s_j} \pi_i(s)\) at \(s = 0\). To prove the lemma, it suffices to show that we can choose \(h_j\) such that 
\[
\partial_{s_j} \pi_i(0) = \delta_{ij},
\]
and each \(h_j\) vanishes in a small neighborhood \(U\) of \(p\).
\smallskip%

One shows (see, e.g., \S\,2.2 in \cite{Honda}) that
\begin{equation}\label{honda}
\Delta \, \partial_{s_j} \psi_{i s} = d\,d^{*}\,h_j \, \psi_i,
\end{equation}
where \(\Delta\) and \(d^{*}\) correspond to the metric \(g\). Define the current \(\delta_{p} \colon \Omega^2 \to \mathbb{R}\) as:
\[
\delta_{p}(\varphi) = \langle \omega, \varphi \rangle \vert_{p}, \quad \text{for each } \varphi \in \Omega^2.
\]
For a general background on currents, see \cite{G-H}. Given our assumptions about \(g\), we get:
\begin{equation}\label{g0-zero}
\delta_{p}(\varphi) = 0, \quad \text{for each } \varphi \in \mathcal{H}^{+}_{g}.
\end{equation}
Let \(C\) be the constant such that
\[
\int_{V} C r^{-2} \Delta f = f(p), \quad \text{for all smooth functions } f \text{ with } \operatorname{supp}f \subset V.
\]
Thus, for any form \(\varphi\) with \(\operatorname{supp}\varphi \subset V\), we get:
\[
\int_{V} \langle C r^{-2} \omega, \Delta \varphi \rangle = \delta_{p}(\varphi).
\]
Let \(V' \subset V\) be a smaller neighborhood of \(p\), and let \(\rho\) be a smooth function that equals 1 in \(V'\) and 0 outside \(V\). Letting 
$\hat{\omega}$ be $\rho\, C r^{-2} \omega$, we define the current $T_{\hat{\omega}}$ by
\[
T_{\hat{\omega}}(\varphi) = \int_{M} \langle \hat{\omega}, \varphi \rangle\quad\text{for each $\varphi \in \Omega^2$.}
\]
$\Delta$ acts on $T_{\hat{\omega}}$ as follows:
\[
\Delta\, T_{\hat{\omega}}(\varphi) = T_{\hat{\omega}}(\Delta\, \varphi).
\]
Consider the 2-form equal to \(\Delta (C \rho r^{-2} \omega)\) on \(X - p\) and 0 at \(p\). We denote this form by \(\Delta\, \hat{\omega}\). For each \(\varphi \in \Omega^2\), we get:
\[
\Delta\, T_{\hat{\omega}}(\varphi) = \delta_p(\varphi) + \int_{M} 
\langle \Delta\, \hat{\omega}, \varphi \rangle.
\]
\eqref{g0-zero} implies that $\Delta\, \hat{\omega} \perp \calh^{+}_{g}$. 
Since $\Delta\, \hat{\omega}$ is SD, it follows that there exists 
a smooth SD 2-form $\theta$ such that 
$\Delta\, \theta = \Delta\, \hat{\omega}$. We thus get the following equality:
\[
\delta_p(\varphi) = T_{\hat{\omega}}(\Delta\,\varphi) - 
\int_{M} \langle \theta, \Delta\,\varphi \rangle \quad\text{for each $\varphi \in \Omega^2$}.
\]
To simplify notation, in what follows 
we set $C = 1$ and write $\partial_{j}$ instead of $\partial_{s_j}$. Substituting 
$\partial_{j} \psi_{i s}$ into $\varphi$ and using \eqref{honda}, we get:
\[
\delta_p(\partial_{j} \psi_{i s}) = \partial_{j} \pi_i(0) = 
T_{\hat{\omega}}(d\,d^{*}\,h_j\,\psi_i) - 
\int_{M} \langle \theta, d\,d^{*}\,h_j\,\psi_i \rangle = 
\int_{M} \langle \rho\, r^{-2} \omega, d\,d^{*}\,h_j\,\psi_i \rangle - 
\int_{M} \langle \theta, d\,d^{*}\,h_j\,\psi_i \rangle.
\]
If we assume that 
\[
\operatorname{supp} h_j \cap V \subset V'
\]
and that \(h_j\) vanishes near \(p\), we can rewrite the last expression as:
\begin{equation}\label{Theta-1}
\partial_{j} \pi_i(0) = 
\int_{V} \langle d\,d^{*}\, r^{-2} \omega, h_j\,\psi_i \rangle - 
\int_{M} \langle d\,d^{*}\,\theta, h_j\,\psi_i \rangle.
\end{equation}
The first term on the right-hand side is local, while the second is global. We will arrange for the local term to dominate, allowing us to extend the argument to the case \(n > 1\). 
The 2-form \(d\,d^{*}\, r^{-2} \omega\) is $\UU(2)$-invariant. Moreover, for any \(\mu \leq 1\), the pullback of \(d\,d^{*}\, r^{-2} \omega\) by the scaling map \(x \to \mu \cdot x\) equals \(\mu^{-2} d\,d^{*} \, r^{-2} \omega\). Hence, \(d\,d^{*} \, r^{-2} \omega = r^{-4} \gamma\), where the 2-form \(\gamma\) has constant length and is ASD. Let \(\hat{V} \subset V'\) be the subset of \(V\) where the three forms \(\psi_i\) are linearly independent; this is an open, dense subset of \(V'\). On \(\hat{V} \subset V'\), define three fields of endomorphisms \(h^{'}_{j}\) such that \(h^{'}_{j} \psi_i = \gamma \delta_{ij}\). Let \(W \subset \hat{V}\) be an open subset whose closure is compact and contained in \(\hat{V}\). Let \(\chi\) be a non-negative function on \(\hat{V}\) that is supported in \(W\) and positive somewhere. Setting \(h_j = \chi h^{'}_{j}\), we rewrite \eqref{Theta-1} as:
\begin{equation}\label{Theta-2}
\partial_{j} \pi_i(0) = 
\delta_{ij} \int_{W} \chi \langle r^{-4} \gamma, \gamma \rangle - 
\delta_{ij} \int_{M} \chi \langle d\,d^{*}\,\theta, \gamma \rangle.
\end{equation}
Since \(\theta\) is a smooth form on \(M\) with uniformly bounded length, the function \(\langle r^{-4} \gamma - d\,d^{*}\,\theta, \gamma \rangle\) is positive for all small enough \(r\), provided \(\gamma \neq 0\). By choosing \(W\) sufficiently close to the origin, we can arrange for \(\partial_{j} \pi_i(0) = \delta_{ij}\) and ensure that \(W\) does not intersect a small neighborhood \(U\) of \(p\). To complete the proof, it remains to show that \(\gamma\) is non-zero. We compute:
\[
d\,d^{*}\, r^{-2} \omega = - L_{\nabla r^{-2}}\, \omega,
\]
where \(\nabla r^{-2}\) is the gradient vector field of \(r^{-2}\). However, this vector field is not Hamiltonian, and the lemma follows. \qed
\end{proof}
\smallskip%

Let \((M, g)\) be as above, with \(U\) and \(g_s\) as given by Lemma \ref{brun}. We blow up \(M\) at the point \(p\) to obtain \((X, g'_{s})\). Then, construct \(X_{T}\) by applying the neck-stretching deformation to \((X, g'_{s})\) as described in \S\,\ref{blowup}. Next, apply Lemma \ref{blowup-lemma}: associated to \(\psi_{1s}, \psi_{2s}, \psi_{3s}\) on \((M, g_s)\), there are SD harmonic forms \(u_{1 s}, u_{2 s}, u_{3 s}\) on \(X_{T}\) that depend smoothly on \(s\) and \(T\). By Lemma \ref{blowup-lemma}, the forms \(u_{i s}\) are \(C^0\)-close to \(\psi_{i s}\) where both are defined. In particular, if \(\psi_{i s}\) are linearly independent on \(M\), then \(u_{i s}\) are linearly independent on \(X_{T}\) for sufficiently large \(T\). From \eqref{blowup-area}, we get:
\begin{equation}\label{ui}
\int_{C} u_{i s} = A \cdot \pi_i(s)\, e^{-2T} + O(e^{-cT}), \quad \text{for } c > 2.
\end{equation}
For large enough \(T\), the last term is negligible, and thus \eqref{ui} provides a surjective map from a small neighborhood of \(g\) in \(D^3\) onto a small neighborhood of \(0\) in \(\mathbb{R}^3\). This completes the proof.

\smallskip
\section*{Acknowledgements}  VS is supported by National Science Centre, Poland, project number: 2019/35/B/ST1/03573. GS is supported by an SNSF Ambizione fellowship.

\smallskip
\bibliographystyle{plain}
\bibliography{references}

\end{document}